\documentclass[11pt, oneside]{amsart}   	
\usepackage[margin=1.4in,marginparwidth=1in]{geometry}

\usepackage{graphicx}				
\usepackage{amssymb}
\usepackage{xcolor}
\usepackage{pgfplots,tikz}
\usetikzlibrary{patterns, shapes.geometric, intersections, calc}
\definecolor{imayou}{RGB}{154, 154, 235}
\definecolor{usuai}{RGB}{9, 150, 126}
\definecolor{persred}{RGB}{154,63,63}
\definecolor{sand}{RGB}{201, 177, 60}
\pgfplotsset{compat=1.18}
\usepackage{hyperref}
\usepackage{amsmath}
\usepackage{amsthm}
\usepackage{mathrsfs}
\usepackage{mathtools}
\usepackage[title]{appendix}
\mathtoolsset{showonlyrefs=true}
\usepackage{bbm}
\usepackage{enumerate, comment} 
\usepackage[foot]{amsaddr}

\title[Hypoellipticity on time-periodic space-times]{Hypoellipticity on time-periodic space-times}

\theoremstyle{definition}

\theoremstyle{plain}
\newtheorem{theorem}{Theorem}[section]
\newtheorem{lemma}[theorem]{Lemma}

\theoremstyle{definition}

\newtheorem{remark}[theorem]{Remark}

\numberwithin{equation}{section}

\newcommand{\Rb}{\mathbb{R}}
\newcommand{\Zb}{\mathbb{Z}}
\newcommand{\RR}{\mathbb{R}}
\newcommand{\ZZ}{\mathbb{Z}}
\newcommand{\NN}{\mathbb{N}}

\newcommand{\Rd}{\mathbb{R}^d}
\newcommand{\N}{\mathbb{N}}
\newcommand{\C}{\mathbb{C}}

\newcommand{\Dc}{\mathcal{D}}

\newcommand{\M}{\mathcal{M}}
\newcommand{\Nb}{\mathbb{N}}
\newcommand{\Tb}{\mathbb{T}}

\newcommand{\ra}{\rightarrow}

\renewcommand{\>}{\right\rangle}

\newcommand{\limk}{\lim_{k \ra \infty}}

\newcommand{\nm}[1]{\left\| #1 \right\|}

\newcommand{\lp}[2]{ \nm{#1}_{L^{#2}}}

\newcommand{\hp}[2]{\nm{#1}_{H^{#2}}}

\newcommand{\ltwo}[1]{\lp{#1}{2}}

\newcommand{\p}{\partial}

\newcommand{\T}{\mathbb{T}}
\newcommand{\Ac}{\mathcal{A}}

\newcommand{\ti}{\widetilde}

\newcommand{\Op}{\text{Op}}

\newcommand{\Sb}{\mathbb{S}}

\renewcommand{\Im}{\text{Im }}

\newcommand{\abs}[1]{|#1|}

\newcommand{\Qc}{\mathcal{Q}}

\newcommand{\ep}{\epsilon}

\newcommand{\ang}[1]{{\langle{#1}\rangle}}
\newcommand{\noqed}{\renewcommand{\qedsymbol}{}}

\newcommand{\cls}{\mathrm{cl}}

\begin{document}

\author{Sandro Coriasco}
\address{Dipartimento di Matematica ``G. Peano'', Universit\`a di Torino, V. C. Alberto, n. 10, I-10126, Torino, Italy}
\email{sandro.coriasco@unito.it}

\author{Perry Kleinhenz}
\address{Department of Mathematics, Illinois State University, 300 S. School St., Normal IL 61761, USA}
\email{pbklein@ilstu.edu}

\author{Jared Wunsch}
\address{Department of Mathematics, Northwestern University, 2033 Sheridan Rd., Evanston IL 60208, USA}\email{jwunsch@math.northwestern.edu}

\begin{abstract}
    We study the hypoellipticity of operators on a product type Lorentzian manifold where the time variable is periodic.
    In particular, we prove that hypoellipticity holds for such time-periodic equations, with a stronger estimate when the mass parameter or time period lies outside a set of arbitrarily small measure. We consider both compact and noncompact spatial manifolds and provide explicit examples involving the wave operator. The proof relies on a Fourier series decomposition and asymptotics for eigenvalue counting functions.
\end{abstract}

\maketitle

\section{Introduction}\label{sec:intro}
For $M$ a compact Riemannian manifold with metric $g$, let $P$ be a nonnegative operator on $M$. Assume that $P$ satisfies an estimate of subelliptic type, $\ltwo{(1+P)^k u} \approx \hp{u}{\eta k}$ for some $\eta>0$; consequently, by the spectral theorem, the eigenfunctions of $P$ form a complete orthonormal basis of $L^2(M)$, and have eigenvalues tending to $+\infty$. Let  $N_P(z) := \# \{\lambda \in \sigma(P): \lambda \leq z\}$ denote the associated eigenvalue counting function.
For $s \in \N, m\in \Rb, \alpha \in \Rb$, we consider the following inhomogeneous equation:
\begin{equation}\label{inhomog}
    (L- m)u=(D_t^s + \alpha P - m)u = f
\end{equation}
on $\mathcal{M}=\Sb^1 \times M$, a product type Lorentzian manifold, where $t\in\Sb^1=\RR/2 \pi \ZZ$ and $D_t=-i\p_t$. As the name of the $t$ variable suggests, the reader is encouraged to think of \eqref{inhomog} as an evolution equation (e.g., heat or wave equation); the important feature here is the \emph{time-periodicity}. 
For $\M$ a compact manifold, we say $L-m$ is hypoelliptic if for every $u \in \Dc'(\M)$, $(L-m) u \in C^{\infty}(\M)$ implies $u \in C^{\infty}(\M)$. (In the noncompact case, we employ tempered distributions instead of $\Dc'(\M)$.) In this paper, we explore the genericity of hypoellipticity of such time-periodic 
operators in two different senses.  

First, we show that hypoellipticity is typical when we fix $\alpha=-1$ and vary the 
mass parameter $m$. 
\begin{theorem}\label{thm1}
Assume that for some $\gamma,\rho>0$, $N_P(z) = C z^{\gamma} + O (z^{\gamma-\rho})$.
Fix $\alpha=-1$ and $\beta>\gamma s+\max(0,1-\rho s)$, then 
\begin{enumerate}
    \item 
     For almost every $m \in[-1,1]$ there exists $C_m>0$ such that 
\begin{equation}
    \nm{u}_{L^2_{t,x}(\M)} \leq C_m \nm{(L -m) u}_{H^{\beta}_t L^2_x(\M)}.
\end{equation}

    \item For all $\ep>0$, there exists a set $\mathcal{A}_{\ep} \subset [-1,1]$ with $\mu(\mathcal{A}_\ep) <\ep$ and $C_{\ep}>0$ such that for all $m \in [-1,1] \backslash \mathcal{A}_{\ep}$, we have the uniform upper bound
    \begin{equation}
        \nm{u}_{L^2(\M)} \leq C_{\ep}\nm{(L-m) u}_{H^{\beta}_t L^2_x(\M)} .
    \end{equation}
\end{enumerate}
\end{theorem}
The above estimates (and those following) should be interpreted as meaning that finiteness of the right-hand side yields finiteness of the left-hand side with the corresponding bound, i.e. here, $(L-m) u \in H^\beta L^2$ implies $u \in L^2$.  Since $L$ commutes with powers of $D_t^2+P$, it is easy to see that the estimate implies global hypoellipticity.

Our second result is that hypoellipticity holds generically for $L$ as we vary $\alpha$. Note that varying $\lvert\alpha\rvert$ is equivalent to varying the time-period, i.e.\ the length of $\Sb^1_t$. Note also for $s$ even and $\alpha>0$, the following result is trivial, as the operator is elliptic.

\begin{theorem}\label{thm2}
    
Assume that for some $\gamma>0$, $N_P(z)=O(z^{\gamma})$. Fix $\beta>\max(s(\gamma-1)+1,0)$. Then 
    \begin{enumerate}
        \item For almost every $\alpha \in \Rb \backslash (-1,1)$, there exists $C_{\alpha}>0$ such that 
        \begin{equation}
            \nm{u-\Pi_0 u}_{L^2_{t,x}(\M)} \leq C_\alpha \nm{L u}_{H^{\beta}_t L^2_x(\M)},
        \end{equation}
        where $\Pi_{0}$ is the projection onto $\ker P \cap \ker D_t$.
        \item For all $\ep>0$, there exists a set $\mathcal{B}_{\ep}$ with $\mu(\mathcal{B}_{\ep})<\ep$ and $C_{\ep}>0$ such that for all $\alpha \in ([-2,-1] \cup [1,2]) \backslash \mathcal{B}_{\ep}$ we have the uniform upper bound
        \begin{equation}
            \nm{u- \Pi_{0} u}_{L^2(\M)} \leq C_{\ep} \nm{L u}_{H^{\beta}_t L^2_x(\M)},
        \end{equation}
        where $\Pi_{0}$ is the projection onto $\ker P \cap \ker D_t$.
    \end{enumerate} 
\end{theorem}
\begin{remark}
    For an explicit example satisfying the hypotheses of these theorems, take $s=2$ and $P=-\Delta_g$, where $\Delta_g$ is the negative-definite Laplace-Beltrami operator for the metric $g$. Then when $\alpha=-1$, $-L=\Box=\p_t^2-\Delta_g$ is the \emph{wave operator}. 
    Note that $-\Delta_g$ satisfies a Weyl law, in particular if $M$ has dimension $d$ then
    \begin{equation}
        N_P(\lambda) = \lambda^{\frac{d}{2}} (2\pi)^{-d} \omega_d \operatorname{vol}(M) + O(\lambda^{\frac{d-1}{2}}),
    \end{equation}
    where $\omega_d$ is the volume of the unit ball in $d$ dimensions.     
    So the assumptions on $N_P$ are satisfied with $\gamma=\frac{d}{2}$, $\rho=\frac{1}{2}$. Thus for Theorem \ref{thm1} we must take $\beta>d$ and for Theorem \ref{thm2} we must take $\beta>d-1$.

    For examples of operators $P$ defined on certain noncompact manifolds, to which (a small variant of) Theorem \ref{thm1} and Theorem \ref{thm2} extend,
    see Section \ref{sec:aEmf} below (see also \cite[Sections 2, 4, and Appendix]{AvilaBoninoCoriasco2026} and
\cite[Sections 2 and 3]{AvilaBoninoCoriasco2025b}).
\end{remark}


%
To further discuss the context and proof of our result, let $\phi_k$ denote the orthonormal basis of eigenfunctions of $P$. We may then write 
\begin{align}
    u(x,t) = \sum_{j=-\infty}^\infty \sum_{k=0}^{\infty} a_{j,k} e^{ijt} \phi_k(x), \\
    f(x,t) = \sum_{j=-\infty}^\infty \sum_{k=0}^{\infty} b_{j,k} e^{ijt} \phi_k(x).
\end{align}
Then, \eqref{inhomog} implies
\begin{equation}\label{ab}
    a_{j,k} (j^s+\alpha\lambda_k-m)= b_{j,k}.
\end{equation}
Equalities of this form can be obtained in a variety of related contexts, and these naturally give rise to non-Liouville conditions, 
also called Diophantine conditions, such as (with $s=1$)
\begin{equation}\label{eq:Diophantine}
    |j+\alpha \lambda_k| \geq C k^{-B}
    \text{ or }
    |j+\alpha \lambda_k| \geq C_\varepsilon \exp(-\varepsilon k^B)
    .
\end{equation}
In many such contexts these Diophantine conditions are equivalent to hypoellipticity (see, e.g.,  \cite{Bergamasco1994,Bergamasco1999,BergamascoCordaroMalagutti1993,BergamascoDattoriGonzalez2018,AvilaBoninoCoriasco2026, AvilaCappiello2022,Hounie1979} and the references quoted therein). However, it is unclear from these previous results if an $\alpha$ can actually be chosen to ensure that \eqref{eq:Diophantine} holds. Our second result addresses this: when $P$ satisfies the appropriate spectral hypotheses, such an $\alpha$ is typical.


\smallskip



\subsection{Literature Review}\label{subs:litrev}
The two basic questions for an evolution operator of the simplest form, namely $L=D_t+\alpha P$, as well as for vector fields and/or systems, are:
\begin{itemize}
\item \emph{global hypoellipticity}: for a distribution $u$ in the chosen linear functional class $H'$, dual to the \textit{regular, test-functions space} $H$, if $Lu \in H$ do we have $u \in H$ as well?
\item \emph{global solvability}: for which (compatible) right-hand sides $f$ does there exist $u \in H'$ solving $Lu=f$ (possibly, with a controlled loss of regularity)?
\end{itemize}
The answers depend on various inputs: the growth of the eigenvalue sequence $(\lambda_k)_k$, the topology of the test-functions $H$ and corresponding distributions $H'$, the sign and real/imaginary parts of the coefficient $\alpha$, and the lower bound imposed in \eqref{eq:Diophantine}.  
The prototypical non-Liouville bounds are:
\begin{itemize}
\item polynomial, that is of type $|j+\alpha\lambda_k|\ge Ck^{-B}$, $H=$ smooth or Sobolev spaces (see, e.g., \cite{Bergamasco1994,BergamascoDattoriGonzalez2017,AvilaBoninoCoriasco2025b,AvilaBoninoCoriasco2026,Hounie1979});
\item subexponential, that is of type $|j+\alpha\lambda_k|\ge C_\varepsilon \exp(-\varepsilon k^B)$,
$\varepsilon > 0$, $H=$ Gevrey and Gelfand--Shilov spaces (see, e.g., \cite{Bergamasco1999,BergamascoCavalcantiGonzalez2021,BergamascoDattoriGonzalez2018,AvilaCappiello2022,AvilaCappiello2025Solvability});
\item Komatsu/Denjoy--Carleman, that is of type
$|j+\alpha\lambda_k|\ge C \exp(-M( k))$, $H=$ analytic or Denjoy--Carleman ultradifferentiable classes, where the function $M$ is determined by the weight sequence of the class (see, e.g.,  \cite{VictorArias2021DenjoyCarleman}).
\end{itemize}

Results of this type 
go back to Seeley in the 1960s \cite{Seeley1965,Seeley1969} and 
Greenfield and Wallach  in the 1970s 
\cite{GreenfieldWallach1972,GreenfieldWallach1973VectorFields,GreenfieldWallach1973Remarks}. A notable example is the constant coefficient operator/vector field
\[
    L_\alpha=D_t+\alpha D_x, \text{ or }
    X_\alpha=\partial_t+\alpha \partial_x,
    \quad (t,x)\in\mathbb{T}^2.
\]
For $(k,\ell)\in\mathbb{Z}^2$, setting $\varphi_{k\ell}(t,x)=e^{i(kt+\ell x)}$, then $L\varphi_{k\ell}=(k+\alpha \ell)\varphi_{k\ell}$. An irrational $\alpha\in\mathbb{R}$ is non-Liouville
if there exist $C,N>0$ such that 
\[
    |k+\alpha\ell|\ge C|\ell|^{-N}, \quad (k,\ell)\in\mathbb{Z}^2, \ell\not=0,
\]
or, equivalently, $\mathrm{dist}(\alpha\ell,\mathbb{Z})\ge C\ell^{-N}$. 
Notice that a rational $\alpha$ produces infinitely many resonances.
Greenfield and Wallach \cite{GreenfieldWallach1972} proved in this setting that $C^{\infty}/\Dc'$ hypoellipticity 
of $L_\alpha$ holds if and only if 
$\alpha$ is a non-Liouville irrational number. Since Liouville numbers are measure zero in the real numbers, hypoellipticity is typical in this context. This can be considered the basic model of the now often observed 
Diophantine condition for global hypoellipticity. However, in other settings the eigenvalues $\lambda_k$ are not explicitly known, so it is not obvious when such a condition is satisfied. 

Cardoso and Hounie \cite{CardosoHounie1977} developed an abstract setting for studying hypoellipticity, and Hounie \cite{Hounie1979} studied operators of the form 
%
\[
    L = D_t+A(t), \quad t\in\Sb^1,
\]
on a scale of Hilbert spaces generated by a suitable operator $A$. The spectral properties of $A$ do not play a role locally in $t$, but become
crucial in the global periodic setting.
In \cite[Theorem 2.2]{Hounie1979}, taking $A=-i\partial_x$, the
model operator 
\[
    L=\partial_t - b(t)\partial_x, \quad (t,x)\in\mathbb{T}^2,
\]
is considered. When $\Im b \equiv 0$, hypoellipticity is equivalent to a non-Liouville bound on $\bar{b}=(2 \pi)^{-1} \int_0^{2\pi} b(t) dt.$ When $\Im b \not \equiv 0$ hypoellipticity is equivalent to $\Im b$ not changing sign. In the literature similar conditions are referred to as condition $(\mathcal{P})$.
%
%


Through the 1980s--1990s and early 2000s, many substantial contributions were made to the study of
the global theory of vector fields and first-order systems on tori and compact manifolds, by
Bergamasco, Cordaro, Malagutti, Petronilho, and others; see, e.g., \cite{Bergamasco1994,Bergamasco1999,BergamascoCordaroMalagutti1993,BergamascoCordaroPetronilho2004,Petronilho2005} and the references quoted therein. They
studied perturbations, analytic regularity, solvability, and condition $(\mathcal{P})$-type hypotheses. 
A model perturbed operator is of the form
\[
    L=\partial_t+\alpha\partial_x+q(t,x) =X_\alpha+q(t,x),\quad (t,x)\in\mathbb{T}^2,
\]
with $q$ an order $0$ perturbation.  
These authors showed, for broad classes of perturbations, that global hypoellipticity is true under analogous Diophantine conditions or condition $(\mathcal{P})$, while resonant or sign-changing configurations allow the construction of counterexamples or obstructions. 
Dickinson, Gramchev, and Yoshino analyzed first-order pseudo-differential operators on tori using normal forms and a Diophantine condition \cite{DickinsonGramchevYoshino1996}. 




The subsequent literature on hypoellipticity of time-periodic operators is large; we now give a sampling, referring to the reference lists of the quoted papers for further details.
 Bergamasco, Dattori da Silva, Gonzalez, Kirilov, Cavalcanti and others refined the torus theory for complex vector fields, periodic solutions, Gevrey regularity, and solvability \cite{BergamascoCavalcantiGonzalez2021,BergamascoDattoriGonzalez2017,BergamascoDattoriGonzalez2018,BergamascoDattoriGonzalezKirilov2015}. 
Avila, Gramchev and Kirilov \cite{AvilaGramchevKirilov2018} studied operators on $\mathbb{T}\times M$ by separation of variables with respect to a positive elliptic operator on the closed manifold $M$. 
 Kirilov, Moraes, and Ruzhansky \cite{KirilovMoraesRuzhansky2021JFA,KirilovMoraesRuzhansky2021Komatsu}
 studied and characterized global hypoellipticity and global solvability for vector fields on a compact Lie group $G$, also within the framework of the Komatsu classes. 
 Kirilov, Kowacs, and Moraes \cite{KirilovKowacsMoraes2024ToriSpheres} studied operators on product manifolds of the form  $\T^{r+1}\times(\mathbb S^3)^s$. These are evolution-like first-order operators with a zero-order perturbation, involving left-invariant vector fields on the $\Sb^3 \simeq SU(2)$ factors. The necessary and sufficient conditions include Diophantine inequalities, change of sign of coefficient functions, and  properties of level sets.  
 
 Finally, parallel to the above work, the previous development of the spectral analysis of operators defined on noncompact domains (see, e.g., Helffer, Robert, Shubin \cite{Helffer1984,HelfferRobert1981,Shubin2001}; see also Nicola and Rodino \cite{NicolaRodino2010}) made it possible to consider time-periodic operators on noncompact spaces.
 In these contexts, we can recall the work by Avila and Cappiello, involving  elliptic operators $P$ of Shubin type \cite{AvilaCappiello2022,AvilaCappiello2025Solvability}, by Kowacs, on the Schwartz spaces \cite{Kowacs2024Schwartz}, and by 
Kowacs and Tokoro, on time-periodic Gelfand-Shilov spaces via non-discrete Fourier analysis \cite{KowacsTokoro2026NonDiscrete}.
 Avila, Bonino and Coriasco \cite{AvilaBoninoCoriasco2025b,AvilaBoninoCoriasco2026} studied time-periodic evolution operators of the form
 \[
    L=D_t+\omega P, \quad \omega \in\C, P\in\Psi^{r,\rho}_\cls,
 \]
where $\Psi^{r,\rho}_\cls$ denotes the operators (locally) defined by (classical) $SG$-symbols of order $r,\rho\in\mathbb{R}$ (see Section \ref{sec:aEmf} below), 
 within the framework of mixed Gevrey-Sobolev-Kato spaces, characterizing their global hypoellipticity and solvability on $\T\times\Rd$ and on $\T\times X$, with $X$ an asymptotically Euclidean manifold. 

Much of the work described above shows that hypoellipticity holds subject to a spectral hypothesis. Beyond the case of $\Tb^n$, as studied by Greenfield and Wallach \cite{GreenfieldWallach1972}, the authors are unaware of previous work verifying that such a spectral hypothesis holds generically; that is the goal of this paper.

\smallskip

The paper is organized as follows. In  Section \ref{sec:circle}
we first look at the special case when
$M=\Sb^1$. In the subsequent Section \ref{sec:generalM} we give the proofs of our main results for the case of a general compact Riemannian manifold $M$.
In the concluding Section \ref{sec:aEmf} we focus on the case of $P$ being a suitable element of a class of operators defined on certain noncompact Riemannian manifolds.
Namely, there $M$ is (the interior of) an asymptotically Euclidean manifold, and we show that (a small variant of) Theorem \ref{thm1} and Theorem \ref{thm2} extend to such a geometric setting.

\textbf{Acknowledgments:} The authors would like to thank Adalberto Bergamasco, Paulo Cordaro, Pedro Tokoro, Ruoyu Wang, and Gaywalee Yamskulna for helpful conversations. The first author has been partially supported by the Italian Ministry of the University and Research - MUR, within the framework of the Call relating to the scrolling of the final rankings of the PRIN 2022 - Project Code 2022HCLAZ8, CUP D53C24003370006 (PI A. Palmieri, Local unit Sc. Resp. S. Coriasco). The first author also expresses
gratitude for the hospitality and support extended to him during his visit to the Department of Mathematics, Northwestern University, Evanston (USA), during A.Y. 2025/2026,
where part of this work was developed. The second author acknowledges partial support from NSF grant DMS-2530465 and a grant from the Illinois State University College of Arts and Sciences. The third author acknowledges partial support from NSF grants DMS-2054424 and DMS-2452331 and from Simons Foundation grant MPS-TSM-00007464.

AI tools (Claude and ChatGPT) were used in the proofreading of this manuscript, but the work and writing are the authors' alone.

\section{The circle}\label{sec:circle}
As an amusing warm-up, we consider the special case where $s=2$, $\alpha=-1$, $M=\Sb^1$, the unit circle, and $P=-\partial_x^2$.
\begin{lemma}
    If $\mathcal{M}=\Sb^1_t \times\Sb^1_x$ then for any $m \in (0,1)$
    \begin{equation}
        \ltwo{u} \leq \frac{1}{\nu}\ltwo{(\partial_t^2 -\partial_x^2+m) u},
    \end{equation}
    where $\nu=\frac{1}{2}-\abs{\frac{1}{2}-m}=\inf\{\abs{m-j},\ j \in \ZZ\}$.
\end{lemma}
\begin{proof}
    Since $M=\Sb^1$, we may enumerate the eigenvalues $\lambda_k$ with multiplicity as $\ell^2$, $\ell \in \ZZ$, then for $m \in (0,1)$ we have (abusing notation slightly by switching $a_{j,\bullet},\ b_{j,\bullet}$ to now have the second index in $\ZZ$) 
\begin{equation}
    a_{j,\ell} = \frac{1}{\ell^2-j^2+m} b_{j,\ell} \implies |a_{j,\ell}| \leq \frac{1}{\nu}|b_{j,\ell}|, \forall j,\ell \implies \|u\|_{L^2} \leq \frac{1}{\nu} \|f\|_{L^2}.\qed
\end{equation}
\noqed
\end{proof} 

We now turn to the massless case, where we tune the parameter $\alpha$.

Recall that Lebesgue-almost every irrational number $x$ is not a Liouville number, meaning there exist $N_0 \in \NN$ such that
\begin{equation}
    \left|x-\frac{p}{q} \right| \geq q^{-N_0}
\end{equation}
for all $p,q\in \ZZ$, $q >1$.  For non-Liouville numbers $x$, the \emph{irrationality exponent} is defined as the finite number
\[
\mu(x) = \sup\{s: 0<\abs{x-p/q}<1/q^s\text{ for infinitely many }p \in \ZZ,\ q \in \NN\}.
\]
Note that for irrational numbers $x$, the irrationality exponent satisfies $\mu(x) \geq 2$.
\begin{lemma}
    Let $\alpha \in \mathbb{R}$ be an irrational number which is not a Liouville number.  Let $\nu>\mu(\alpha)$. Then there exists $C>0$ such that 
    \begin{equation}
        \ltwo{u-\bar u} \leq C \nm{(\p_t^2-\alpha\p_x^2) u}_{H^{2(\nu-1)}_x L^2_t},
    \end{equation}
    where $\bar u$ denotes the average of $u$ over the torus.
\end{lemma}
\begin{proof}
    Since $\nu>\mu(\alpha)$, there is $\ep>0$ such that for all $p \in \mathbb{Z}$ and $q \in \mathbb{N}$ we have
\begin{equation}
    \left|\alpha-\frac{p}{q} \right| > \ep q^{-\nu}.
\end{equation}
Thus for all $j \in \mathbb{Z}$, $k \in \Zb \backslash\{ 0\}$,  we have 
\begin{equation}
    |j^2-\alpha k^2 | = k^2 \left| \alpha - \frac{j^2}{k^2}\right| > \ep\abs{k}^{-2(\nu-1)} 
\end{equation}
hence
\begin{equation}
    a_{j,k} = \frac{1}{j^2 -\alpha k^2} b_{j,k}, \implies |a_{j,k}| \leq \ep^{-1} |k|^{2(\nu-1)} |b_{j,k}|. 
\end{equation}
On the other hand, for $k=0$, $j \neq 0$, we have the trivial estimate
\[
|a_{j,0}| \leq |b_{j,0}|.
\]
Hence
\[
|a_{j,k}| \leq C \ang{k}^{2(\nu-1)} |b_{j,k}|,\quad \text{ for } (j,k) \neq (0,0), 
\]
and this yields the desired estimate, as subtracting the average eliminates the $(0,0)$-mode.
\end{proof}

\section{General spatial manifold}\label{sec:generalM}
We now turn to the case of a general spatial manifold $M$. 

\subsection{General $M$, shifting mass $m$}\label{s:genManifoldMass}
Here we prove Theorem~\ref{thm1}.

First consider the set 
\begin{equation}
    \mathcal{R} = \{j^s -\lambda_k : j \in \ZZ, \lambda_k \in \sigma(P) \}.
\end{equation}
It is countable and so $\mu(\mathcal{R})=0$.

For $j \in \Zb$ with $|j| \geq 2$ consider the interval $((|j|-2)^s,(|j|+2)^s]$. There are 
\begin{equation}
    N_P( (|j|+2)^s) - N_P((|j|-2)^s) = C_M (|j|+2)^{s \gamma } - C_M (|j|-2)^{s\gamma} + O(|j|^{s\gamma-\rho s}) \leq C|j|^{s\gamma-\min(1,\rho s)}
\end{equation}
eigenvalues of $P$ in this interval.
Let
\begin{equation}
    \mathcal{U}_j = \{ z \in [(|j|-1)^s, (|j|+1)^s]: \exists \lambda_k \in \sigma(P), \text{ such that } |\lambda_k-z|<\<j\>^{-\beta}\},
\end{equation}
where $\<j\>=(1+|j|^2)^{1/2}$.
Since $\mathcal{U}_j$ is composed of at most $C |j|^{s\gamma-\min(1,\rho s)}$ neighborhoods of length $2\ang{j}^{-\beta}$, we have
$\mu(\mathcal{U}_j) \leq C_M \<j\>^{s\gamma-\min(1,\rho s)-\beta}$. Then define
\begin{equation}
    A_j = \{m \in [-1,1]: j^s-m \in \mathcal{U}_j\},
\end{equation}
and we have, since $\beta>s \gamma+\max(0, 1-\rho s)$
\begin{equation}
    \sum_{j=2}^{\infty} \mu(A_j) \leq \sum_{j=2}^{\infty} C_M |j|^{s\gamma-\min(1,\rho s)-\beta}<\infty.
\end{equation}

1) By the Borel--Cantelli lemma 
\begin{equation}
    \mu(\limsup_{n \rightarrow \infty} A_n) = \mu \big( \bigcap_{n=2}^{\infty} \bigcup_{j =n} A_j \big)=0.
\end{equation}
Thus $\mathcal{G}=[-1,1] \backslash (\mathcal{R} \cup \limsup_{n \ra \infty} A_n)$ has $\mu(\mathcal{G})=2$.
For $m \in \mathcal{G}$, and thus for almost every $m \in [-1,1]$, there exists $N_m \geq 2$ such that
\begin{equation}
j^s-m \notin    \mathcal{U}_j,\ \text{ for all } j \geq N_m.
\end{equation}
Therefore, for all $j \geq N_m$,
\begin{equation}
    |\lambda_k-(j^s-m)| \geq
        \<j\>^{-\beta}, \quad \lambda_k \in \sigma(P). 
\end{equation}
Since $N_m \geq 2$, separately considering $s$ even or odd, we in fact have for all $|j| \geq N_m$ 
\begin{equation}
    |\lambda_k-(j^s-m)| \geq
        \<j\>^{-\beta}, \quad\lambda_k \in \sigma(P). 
\end{equation}
Likewise for $|j|\leq N_m$, provided $\lambda_k^{1/s}>N_m+2$ we also have
\[
 |\lambda_k-(j^s-m)| \geq 1.
\]
This certainly holds for $k$ large because there are only finitely many $\lambda_k \leq (N_m+2)^s$.
Thus for $N_m'$ sufficiently large,
\begin{equation}
    |\lambda_k-(j^s-m)| \geq \ang{j}^{-\beta}, \quad \max(|j|, k)\geq N_m'.
\end{equation}
For $\max(|j|,k) < N_m'$, since $m \not \in \mathcal{R}$ we have 
\begin{equation}
    |\lambda_k-(j^s-m)| = c_{j,k}>0.
\end{equation}
Then let 
\begin{equation}
    c_{\min} = \min(1, c_{j,k}) >0,
\end{equation}
where we take the minimum  over the finite set $\{(j,k) \in \ZZ \times \NN_0:\max(|j|,k) < N_m'\}$. 
Then for all $j \in \ZZ, k \in \NN_0$ we have 
\begin{equation}
    |\lambda_k-(j^s-m)| \geq c_{\min} \ang{j}^{-\beta}.
\end{equation}
Equivalently in the notation of \eqref{ab} 
\begin{equation}
    |a_{j,k}| \leq c_{\min}^{-1} \<j\> ^{\beta} |b_{j,k}|, \quad \text{ for all } j \in \ZZ,k \in \NN_0. 
\end{equation}
Therefore we obtain
\begin{equation}
    \ltwo{u} \leq c_{\min}^{-1} \| f\|_{H^{\beta}_t L^2_x}. 
\end{equation}
as desired.


2) Note that without loss of generality we may choose $\epsilon$ to be arbitrarily small. Now choose $n \geq 2$ large enough so that
\begin{equation}
    \sum_{j=n}^{\infty} \mu(A_j) <\frac{\ep}{2},
\end{equation}
then let $\tilde{\mathcal{A}}_{\ep} = \bigcup_{j \geq n} A_j$.
There are finitely many eigenvalues in $[0,n^s+2]$, so the set 
\begin{equation}
    \ti{A} = \left\{m \in \mathbb{R}: m=l^s-\lambda_k, \text{ for some } l \in \Zb, |l|\leq n,\ \lambda_k \in \sigma(P),\ \lambda_k \leq n^s+2 \right\},
\end{equation}
is finite. Replacing each point in the set by a ball of radius $\frac{\ep}{4(N+1)}$ centered at the point, where $N$ is the number of points in $\ti{A}$, we obtain a set $\ti{A}'$ of measure $\mu(\ti{A}')< \frac{\ep}{2}$.

Consequently, if
\begin{equation}
    \mathcal{A}_{\ep}= \ti{\Ac}_{\ep} \cup \ti{A}'
\end{equation}
then $\mu(\Ac_{\ep}) <\ep$.

Now for any $m \in [-1,1] \backslash \mathcal{A}_{\ep}$ we have
\begin{align}
    |\lambda_k-(j^s-m)| \geq \begin{cases}    \<j\> ^{-\beta}, & |j| \geq n,\\
      \frac{\ep}{4(N+1)}, & |j| \leq n-1, \lambda_k \leq n^s+2, \\
      1, & |j| \leq n-1, \lambda_k \geq n^s+2.
     \end{cases}
\end{align}
Let $c_{\min}=\min( \frac{\ep}{4(N+1)},1)$, and then since $c_{\min}, \<j\>^{-\beta} \leq 1$  we have 
\begin{equation}
    |\lambda_k-(j^s-m)| \geq  c_{\min} \<j\>^{-\beta}.
\end{equation}
Therefore for all $j \in \mathbb{Z}$ and $k \in \mathbb{N}_0$ in the notation of \eqref{ab}
\begin{equation}
    |a_{j,k}| \leq \frac{1}{c_{\min}}\<j\>^{\beta} |b_{j,k}|.
\end{equation}
Since $c_{\min}$ depends only on $\ep$, and not on $m$, we have as desired
\begin{equation}
    \ltwo{u} \leq C_{\ep} \nm{f}_{H^{\beta}_t L^2_x}.\qed
\end{equation}

\subsection{General $M$, changing parameter $\alpha$}\label{s:changeAlpha}
%

We now prove Theorem~\ref{thm2}.

First consider the set 
\begin{equation}
    \mathcal{Q} = \left\{\frac{-j^s}{\lambda_k} : j \in \ZZ, \lambda_k \in \sigma(P),\lambda_k >0\right\}.
\end{equation}
It is countable and so $\mu(\mathcal{Q})=0$.

Now fix $\beta>0$ and for $j \in \Zb \backslash \{0\}, n \in \mathbb{N}$ let
   \begin{align*}
       B_{j,n} = \left\{
       \phantom{\left|z+ \frac{j^s}{\lambda_k}\right|}
       \right.\hspace*{-1.5cm}
       z \in [-n-1,-n] \cup[n,n+1]\,:\, &\exists \lambda_k \in \sigma(P) \text{ with } \lambda_k>0,     
       \\
       &\left.\text{such that } \left|z+ \frac{j^s}{\lambda_k}\right| < |j|^{-\beta}\lambda_k^{-1} \right\}.
   \end{align*}
    Note that $|z+\frac{j^s}{\lambda_k}| < |j|^{-\beta} \lambda_k^{-1}$ if and only if $|j^s+z \lambda_k| < |j|^{-\beta}$. This occurs only for values of $\lambda_k \in \left[\frac{|j|^s-1}{n+1},\frac{|j|^s+1}{n}\right]$, and there are at most $N_P((|j|^s+1)/n) = O(|j|^{s\gamma})$ such $\lambda_k$'s. Since $\lambda_k^{-1} \approx |j|^{-s}$, the $|j|^{-\beta}\lambda_k^{-1} \approx |j|^{-\beta-s}$ neighborhoods of the values $-j^s \lambda_k^{-1}$ have total measure $O(|j|^{-\beta+s(\gamma-1)})$. Therefore
    \begin{equation}
        \mu(B_{j,n}) = O(|j|^{-\beta+s(\gamma-1)}).
    \end{equation} 
     Note also, when $s$ is even then $B_{j,n} = B_{-j,n}$, and when $s$ is odd then $B_{j,n}=-B_{-j,n}$. Therefore for $\beta>s(\gamma-1)+1$,
    \begin{equation}
        \sum_{|j| \geq 1} \mu(B_{j,n}) \leq \sum_{j=1}^{\infty} C_M j^{s(\gamma-1)-\beta}<\infty.
    \end{equation}
   
    
1)  Now by the Borel--Cantelli lemma, for a fixed $n$ 
    \begin{equation}
        \mu(\limsup_{|p|\ra \infty} B_{p,n}) = \mu\left( \bigcap_{p=1}^{\infty} \bigcup_{|j| \geq p} B_{j,n} \right)=0.
    \end{equation}
    Let $\mathcal{G}_n=\big([-n-1,-n]\cup[n,n+1]\big)\backslash (\Qc \cup \limsup_{|p|\ra \infty} B_{p,n} )$, and note $\mu(\mathcal{G}_n) = 2$.
    For all $\alpha \in \mathcal{G}_n$ there exists $N_{\alpha}>0$ such that for all $|j| \geq N_{\alpha}$ and $\lambda_k \in \sigma(P)$ with $\lambda_k>0$ we have 
    \begin{equation}
         |j^s + \alpha \lambda_k| \geq \<j\>^{-\beta}.
    \end{equation}
    Since $\Rb \backslash (-1,1)= \bigcup_{n=1}^{\infty} [-n-1,-n] \cup [n,n+1]$, and the countable union of measure zero sets is measure zero, the above holds for almost every $\alpha \in \Rb \backslash (-1,1)$. 
    Furthermore, since $|j| \geq N_{\alpha}>0$ we have for all $\lambda_k \in \sigma(P)$
    \begin{equation}
         |j^s + \alpha \lambda_k| \geq \<j\>^{-\beta}.
    \end{equation}
    For $|j|<N_\alpha$ and $k$ sufficiently large, we certainly also have
    \begin{equation}
        | j^s+ \alpha \lambda_k| \geq 1, 
    \end{equation}
    since $\limk \lambda_k= \infty$.
    Therefore for some $N_\alpha' \in \NN$
    \begin{equation}
        |a_{j,k}| \leq \ang{j}^{\beta}|b_{j,k}|, \quad  \text{ when }\max(|j|,k) \geq N_\alpha'.
    \end{equation}
For $\max(|j|,k) < N_{\alpha}'$, $(j,\lambda_k) \neq (0,0)$, since $\alpha \not \in \mathcal{Q}$ we have 
\begin{equation}
    |j^s +\alpha \lambda_k| = c_{j,k}>0.
\end{equation}
Then let 
\begin{equation}
    c_{\min} = \min(1, c_{j,k}) >0,
\end{equation}
where we take the minimum  over the finite set $\{(j,k) \in \ZZ \times \NN_0:\max(|j|,k) < N_{\alpha}', (j,\lambda_k) \neq (0,0)\}$. Then we have
\begin{equation}
    |j^s+\alpha\lambda_k| \geq c_{\min} \ang{j}^{-\beta}, \quad (j,\lambda_k) \neq (0,0).
\end{equation}
Hence, the desired estimate follows, as subtracting $\Pi_0 u$ eliminates the $(j=0, \lambda_k=0)$ modes.

2) Note that without loss of generality we may choose $\epsilon$ to be arbitrarily small. Choose $J \in \NN$ so that
$$
\sum_{|j| \geq J} \mu(B_{j,1})<\frac{\ep}{2}.
$$
Set $\tilde{\mathcal{B}}_\ep=  \bigcup_{|j| \geq J}B_{j,1}$.
There are finitely many eigenvalues in $[0,2J^s]$, so the set 
\[
\mathcal{C}=\left\{z \in \Rb \backslash \left(-\frac{1}{2},\frac{1}{2}\right): z=-l^s/\lambda_k \text{ for some } l \in \ZZ, |l|\leq J,\ \lambda_k \in \sigma(P), \lambda_k >0 \right\}
\]
is finite. Replacing each point in the set by a ball of radius $\frac{\ep}{4(N+1)}$ centered at the point, where $N$ is the number of points in $\mathcal{C}$, we obtain a set $\mathcal{C}'$ of measure $\mu(\mathcal{C}') < \frac{\ep}{2}$.

So if we set
\[
\mathcal{B}_\ep = \tilde{\mathcal{B}}_\ep\cup \mathcal{C}',
\]
then $\mu(\mathcal{B}_{\ep}) <\ep$. 

Now for any $\alpha \in ([-2,-1] \cup [1,2]) \backslash \mathcal{B}_{\ep}$ we have 
\begin{equation}
    |j^s +\alpha \lambda_k| \geq \begin{cases}
        \ang{j}^{-\beta}, & |j| \geq J, 0<\lambda_k \\
        \lambda_{\min}\frac{\epsilon}{4(N+1)}, & |j| \leq J, 0<\lambda_k \leq 2J^s \\
        1, & |j| \leq J, 2J^s<\lambda_k \\
        1 & |j| \geq 1, \lambda_k=0,
    \end{cases}
\end{equation}
where $\lambda_{\min}$ is the smallest positive eigenvalue of $P$.
Set
\begin{equation}
    c_{\min}=\min\left\{1,\lambda_{\min}\frac{\epsilon}{4(N+1)}\right\} >0.
\end{equation}
Since $c_{\min}, \ang{j}^{-\beta} \leq 1$,  we have
\begin{equation}
    |j^s + \alpha \lambda_k| \geq c_{\min} \ang{j}^{-\beta}, \quad (j,\lambda_k) \neq (0,0).
\end{equation}
Since $c_{\min}$ does not depend on $\alpha$, we have as desired
\begin{equation}
    \ltwo{u-\Pi_0u} \leq C_{\ep} \nm{f}_{H^{\beta}_t L^2_x},\noqed
\end{equation}
since subtracting $\Pi_{0}u$ eliminates the $(j=0, \lambda_k=0)$ modes.\qed
\begin{remark}
    Note that if $\alpha \in [-2n,-n] \cup [n,2n]$, then we can write 
    \begin{equation}
        D_t^s + \alpha P = D_t^s + \frac{\alpha}{n}(nP).
    \end{equation}
    So then $\frac{\alpha}{n} \in [-2,-1] \cup [1,2]$ and $nP$ satisfies the same spectral assumptions as $P$. So we may find a set $\mathcal{B}_{\epsilon/n}$ of measure $\epsilon/n$ in $[-2,-1] \cup [1,2]$, which corresponds to a set $\ti{\mathcal{B}}_{\epsilon}$ of measure $\epsilon$ in $[-2n,-n] \cup [n,2n]$, so that there exists $C_{\epsilon}>0$, such that for all $\alpha \in ([-2n,-n] \cup [n,2n]) \backslash \ti{\mathcal{B}}_{\epsilon}$ we have 
    \begin{equation}
        \ltwo{u -\Pi_0 u} \leq C_{\epsilon} \nm{Lu}_{H^{\beta}_t L^2_x}.
    \end{equation}
\end{remark}

\section{Noncompact spatial manifolds: the case of asymptotically Euclidean manifolds}\label{sec:aEmf}
As we have seen in the previous Section \ref{sec:generalM}, our results primarily rely on asymptotics for the eigenvalue counting function $N_P$ of the operator $P$.
This allows us to extend Theorems \ref{thm1} and \ref{thm2} to situations where $M$ is noncompact.

For $M=\Rb^d$, employing the results in \cite{Helffer1984,HelfferRobert1981,Shubin2001} (see also the other related quotations 
in Section \ref{subs:litrev}, \cite{DollGannotWunsch2018}, and the references therein), 
an example with $s=1$ is given by $L=D_t+\alpha H -m$, involving the harmonic oscillator $H=\frac{1}{2}(-\Delta+|x|^2)\in\Op(\Gamma^2_\cls(\Rb^d))$ 
(that is, $H$ is a classical Shubin operator of order $2$, see, e.g., \cite[Ch. 2, \S 2.1]{NicolaRodino2010}). Since
\[
	N_H(z)=\frac{z^d}{d!} + O(z^{d-1}),\quad z\to+\infty,
\]
then Theorems \ref{thm1} and \ref{thm2} extend to this case with $\mathcal{M}=\Sb^1\times\Rb^d$. 
More generally, for $\kappa \in\Nb$ and any positive, elliptic, self-adjoint, Shubin operator $P\in\Op(\Gamma^\kappa_\cls(\Rb^d))$, 
the Weyl formula reads as
\[
	N_P(z)=Cz^\frac{2d}{\kappa}+O(z^{\frac{2d}{\kappa}-\frac{1}{\kappa}}), \quad z\to+\infty,
\]
and Theorems \ref{thm1} and \ref{thm2} can be applied to any $L$ involving such an operator $P$ as well.
Unfortunately, the Shubin calculus does not have good invariance properties under the action of diffeomorphisms, so this
cannot be directly extended to suitable noncompact manifolds
(see \cite{Krainer2025} for a calculus on an asymptotically conic manifold $X$ which reduces to the Shubin calculus when $X^\circ=\Rb^d$).

Another class of operators satisfying global estimates on $\Rb^d$ are the so-called SG-operators or scattering operators
(see \cite{Cordes1976,Melrose1995,NicolaRodino2010,Parenti1972,Schrohe1987}), 
whose symbols satisfy, for some $r,\rho\in\Rb$, any $\alpha,\beta\in\mathbb{Z}^d_+$ and suitable $C_{\alpha\beta}>0$,
the global estimates, 
\[
	|D^\alpha_x D^\beta_\xi a(x,\xi)|\le C_{\alpha \beta} \langle x \rangle^{r-|\alpha|}\langle \xi \rangle^{\rho-|\beta|}, \quad x,\xi\in\Rb^d.
\]
Classical subclasses are defined in terms of expansions in homogeneous components with respect to $x$ and $\xi$, respectively, and
\textit{compatibility conditions} on the corresponding terms of the two expansions (see, e.g., \cite[Ch. 3, \S 3.2]{NicolaRodino2010} for the precise definition). 
The operators associated with local classical symbols of SG/scattering type, whose space we denote
by $\Psi^{r,\rho}_\cls(X)$, are invariantly defined when $X$ is a manifold with boundary (see \cite{Melrose1995}). We naturally obtain operators in this calculus if we equip the manifold with a Riemannian metric that in a collar neighbourhood of the boundary takes the form
\[
	g=\frac{d\chi^2}{\chi^4}+\frac{h(\chi)}{\chi^2},
\]
where $\chi$ is a boundary defining function for $X$ and the smooth, symmetric $2$-tensor family $h(\chi)$ reduces to a Riemannian
metric at the boundary $\chi=0$; in particular, the Laplacian for such a metric lies in $\Psi^{0,2}_\cls(X)$.
The simplest example of such a Riemannian manifold is the (radial or stereographic) compactification of $\Rb^d$.
In this section we take $M=X^\circ$.

If $P\in\Psi^{r,\rho}_\cls(X)$ is elliptic, positive, self-adjoint, and such that $r,\rho>0$, then, setting $\gamma_0=\frac{d}{\min\{r,\rho\}}$, 
it is known that, for $z\to+\infty$,
\[
	N_P(z)=
	\begin{cases}
		C_{1r\rho}z^{\gamma_0}+O(z^\frac{d}{\max\{r,\rho\}})+O(z^{\gamma_0-\frac{1}{\max\{r,\rho\}}})
		& r\not=\rho,
		\\
		C_{2}z^{\gamma_0}\log z+C_{3}z^{\gamma_0}+O(z^{\gamma_0-\frac{1}{r}}\log z),
		& r=\rho,
	\end{cases}
\]
that is,
\begin{equation}\label{eq:weylsg}
	N_P(z)=
	\begin{cases}
		C_{1r\rho}z^{\gamma_0}+O(z^{\gamma_0-\varepsilon_{r\rho}}),
		& r\not=\rho,
		\\
		C_{2}z^{\gamma_0}\log z+C_{3}z^{\gamma_0}+O(z^{\gamma_0-\varepsilon_{r\rho}}\log z),
		& r=\rho,
	\end{cases}
\end{equation}
where
\begin{equation}\label{eq:defeps}
	0<\varepsilon=\varepsilon_{r\rho}=
	\begin{cases}
		\min\left\{d\left|\dfrac{1}{r}-\dfrac{1}{\rho}\right|,\dfrac{1}{\max\{r,\rho\}}\right\}, & r\not=\rho,
		\\
		\dfrac{1}{r}, & r=\rho,
	\end{cases}
\end{equation}
see \cite{BattistiCoriasco2011,CoriascoDoll2021,CoriascoManiccia2013}. In view of the slightly different form of the Weyl formula when $P$ is a scattering operator, 
compared with the one in the assumptions of Theorem \ref{thm1},
the only point to be re-examined here is the estimate of the number of eigenvalues of $P$ belonging to the interval $((|j|-2)^s,(|j|+2)^s]$.
By straightforward computations, \eqref{eq:weylsg} implies
\begin{align*}
N_{js}&=N_P((|j|+2)^s)-N_P((|j|-2)^s)
\\
&=
\begin{cases}
	O(|j|^{s\gamma_0-1})+O(|j|^{s\gamma_0-s\varepsilon}), & r\not=\rho,
	\\
	O(|j|^{s\gamma_0-1}\log|j|)+O(|j|^{s\gamma_0-1})+O(|j|^{s\gamma_0-s\varepsilon}\log|j|), & r=\rho,
\end{cases}
\end{align*}
that is
\[
N_{js}\le
\begin{cases}
	C_1|j|^{s\gamma_0-\varepsilon'}, 	   & r\not=\rho,
	\\
	C_2|j|^{s\gamma_0-\varepsilon'}\log|j|\le C_2|j|^{s\gamma_0-\varepsilon'+\delta}, & r=\rho,
\end{cases}
\] 
with $\varepsilon'=\min\{s\varepsilon,1\}$, $\delta>0$ arbitrarily small. Then, with the same notation employed in the proof of Theorem \ref{thm1},
\[
	\mu(\mathcal{U}_j)\le
	\begin{cases}
		C\langle j\rangle^{-1+(s\gamma_0+1-\varepsilon'-\beta)}, & r\not=\rho,
		\\
		C\langle j\rangle^{-1+(s\gamma_0+1-\varepsilon'+\delta-\beta)}, & r=\rho.		
	\end{cases}
\]
Notice that, choosing $\beta>s\gamma_0+1-\varepsilon'$, we can always find $\delta>0$ suitably small such that $s\gamma_0+1-\varepsilon'+\delta-\beta<0$.
Then, following the same argument as in Section \ref{s:genManifoldMass}, we have proved the following slight variant of Theorem \ref{thm1}.

{
\renewcommand{\thetheorem}{\ref{thm1}'} 
\begin{theorem}
Let $M$ be the interior of an asymptotically Euclidean manifold $X$, $\mathcal{M}=\Sb^1\times M$, $s\in\Nb$, 
$L$ as in \eqref{inhomog}, 
and $P\in\Psi^{r,\rho}_\cls(X)$ be elliptic, positive and self-adjoint, with order components $r,\rho>0$. Moreover,
let $\gamma_0=\frac{d}{\min\{r,\rho\}}$ and $\varepsilon$ be given by \eqref{eq:defeps}.
Fix $\alpha=-1$ and 
\begin{align*}
	&\beta>\gamma_0 s, \phantom{ + 1- s\varepsilon}\hspace*{5pt}\quad\text{if } s\varepsilon \ge1, \quad \text{or}
	\\
	&\beta>\gamma_0 s + 1- s\varepsilon,\quad \text{if } s\varepsilon<1.
\end{align*}
Then: 
\begin{enumerate}
    \item 
     For almost every $m \in[-1,1]$ there exists $C_m>0$ such that 
\begin{equation}
    \nm{u}_{L^2_{t,x}(\M)} \leq C_m \nm{(L -m) u}_{H^{\beta}_t L^2_x(\M)}.
\end{equation}

    \item For all $\ep>0$, there exists a set $\mathcal{A}_{\ep}$ with $\mu(\mathcal{A}_\ep) <\ep$ and $C_{\ep}>0$ such that for all $m \in [-1,1] \backslash \mathcal{A}_{\ep}$, we have the uniform upper bound
    \begin{equation}
        \nm{u}_{L^2(\M)} \leq C_{\ep}\nm{(L-m) u}_{H^{\beta}_t L^2_x(\M)} .
    \end{equation}
\end{enumerate}
\end{theorem}
\addtocounter{theorem}{-1}
}
Concerning the difference in the hypotheses in Theorem \ref{thm2}, again by \eqref{eq:weylsg}, we have
\[
	N_P(z)=
	\begin{cases}
	O(z^{\gamma_0}), & r\not=\rho,
	\\
	O(z^{\gamma_0+\delta}), & r=\rho,
\end{cases}
\]
with $\delta>0$ arbitrarily small. Arguing as above, if $s(\gamma_0-1)+1<0$, we can find $\delta>0$ suitably small such that 
we also have $s(\gamma_0+\delta-1)+1<0$, so the lower bound for $\beta$ in this case is still $0$. On the other hand, if
$s(\gamma_0-1)+1\ge0$ and $\beta>s(\gamma_0-1)+1$, again there exists $\delta>0$ suitably small such that
$\beta>s(\gamma_0+\delta-1)+1$, so the lower bound for $\beta$ in this case is still $s(\gamma_0-1)+1$.
By the same argument as in Section \ref{s:changeAlpha} the following variant of Theorem \ref{thm2} holds.
{
\renewcommand{\thetheorem}{\ref{thm2}'} 
\begin{theorem}
Let $M$ be the interior of an asymptotically Euclidean manifold $X$, $\mathcal{M}=\Sb^1\times M$, $s\in\Nb$, 
$L$ as in \eqref{inhomog}, 
and $P\in\Psi^{r,\rho}_\cls(X)$ be elliptic, positive and self-adjoint, with order components $r,\rho>0$. Moreover,
let $\gamma_0=\frac{d}{\min\{r,\rho\}}$ and fix $\beta>\max(s(\gamma_0-1)+1,0)$.
Then: 
    \begin{enumerate}
        \item For almost every $\alpha \in \Rb \backslash (-1,1)$, there exists $C_{\alpha}>0$ such that 
        \begin{equation}
            \nm{u-\Pi_0 u}_{L^2_{t,x}(\M)} \leq C_{\alpha} \nm{L u}_{H^{\beta}_t L^2_x(\M)},
        \end{equation}
        where $\Pi_{0}$ is the projection onto $\ker P \cap \ker D_t$.
        \item For all $\ep>0$, there exists a set $\mathcal{B}_{\ep}$ with $\mu(\mathcal{B}_{\ep})<\ep$ and $C_{\ep}>0$ such that for all $\alpha \in ([-2,-1] \cup [1,2] )\backslash \mathcal{B}_{\ep}$ we have the uniform upper bound
        \begin{equation}
            \nm{u- \Pi_{0} u}_{L^2(\M)} \leq C_{\ep} \nm{L u}_{H^{\beta}_t L^2_x(\M)},
        \end{equation}
        where $\Pi_{0}$ is the projection onto $\ker P \cap \ker D_t$.
    \end{enumerate} 
\end{theorem}
\addtocounter{theorem}{-1}
}
As an example where Theorems \ref{thm1}' and \ref{thm2}' apply with $s=2$, consider, on $\Rb^d$, the operator
$P=(1+|x|^2)^\frac{r}{4}(1-\Delta)(1+|x|^2)^\frac{r}{4}$, $r>0$, so that $P\in\Psi^{r,2}_\cls$ is elliptic, positive
and self-adjoint with positive order components. $L$ is then again an operator of Klein-Gordon or wave type, namely
\[
	L = -\partial_t^2+\alpha (1+|x|^2)^\frac{r}{2}(1-\Delta) + \mathrm{l.o.t.}
\]
globally defined on $\Sb^1\times\Rb^d$.
For $r=2$, we find $\gamma_0=\frac{d}{2}$, $\varepsilon=\frac{1}{2}$, and $\varepsilon'=1$. Then,
also in this case, we have to choose $\beta>d$ in Theorem \ref{thm1}' and $\beta>d-1$ in Theorem \ref{thm2}'.

Finally, we note that the hypoellipticity obtained in this noncompact case is a strong version adapted to the noncompact setting: since $P$ has positive spatial order, $(D_t^2+P)^k u \in L^2$ for all $k \in \mathbb{N}$ is equivalent to $u \in \dot{\mathcal{C}}^\infty(\M)$, where this denotes the space of all smooth functions on the compact manifold $\Sb^1 \times X$ (recall that $M=X^\circ$) vanishing to infinite order at the boundary.  In the case when $M$ is the radial compactification of Euclidean space, these are just smooth (in $t$) families of Schwartz functions in $x$.  The hypoellipticity we obtain from the two results above by commuting $(D_t^2+P)^k$ through our estimate is thus $u \in (\dot{C}^{\infty})'(M)$, $(L-m)u \in \dot{\mathcal{C}}^\infty(\M) $, resp.\ $Lu \in \dot{\mathcal{C}}^\infty(\M) $,  implies $u \in \dot{\mathcal{C}}^\infty(\M)$. 
\bibliographystyle{plain}
\bibliography{hypoelliptic}
\end{document}